\documentclass[12pt,twoside]{amsart}
\usepackage{amsmath, amsthm, amscd, amsfonts, amssymb, graphicx}
\usepackage{enumerate}
\usepackage[colorlinks=true,
linkcolor=blue,
urlcolor=cyan,
citecolor=red]{hyperref}
\usepackage{mathrsfs}
\newtheorem{theorem}{Theorem}[section]
\newtheorem{lemma}{Lemma}[section]
\newtheorem{remark}{Remark}[section]

\newtheorem{corollary}{Corollary}[section]
\newtheorem{example}{Example}[section]
\newtheorem{proposition}{Proposition}[section]
\numberwithin{equation}{section}

\begin{document}

\title{Operator inequalities via accretive  transforms}
\author{Ibrahim Halil G\"um\"u\c s,  Hamid Reza Moradi and Mohammad Sababheh}
\maketitle

\begin{abstract}
In this article, we employ certain properties of the transform $C_{M,m}(A)=(MI-A^*)(A-mI)$ to obtain new inequalities for the bounded linear operator $A$ on a complex Hilbert space $\mathcal{H}$. In particular, we obtain new relations among $|A|,|A^*|,|\mathfrak{R}A|$ and $|\mathfrak{I}A|$. Further numerical radius inequalities that extend some known inequalities will be presented too.
\end{abstract}
\pagestyle{myheadings}
\markboth{\centerline {}}
{\centerline {}}
\bigskip
\bigskip
\section{Introduction}
While studying inequalities of Kantorovich type, Dragomir \cite{1} defined the  transform $C_{M,m}:\mathcal{B}(\mathcal{H})\to \mathcal{B}(\mathcal{H})$ by
$$C_{M,m}(A)=(MI-A^*)(A-mI),$$ where $M>m>0$ are predefined real numbers, $\mathcal{B}(\mathcal{H})$ is the $C^*-$algebra of all bounded linear operators on a complex Hilbert space $\mathcal{H}$, $I$ is the identity operator in $\mathcal{B}(\mathcal{H})$ and $A^*$ is the adjoint of $A\in\mathcal{B}(\mathcal{H})$.

Basic properties and applications of $C_{M,m}$ were presented in \cite{1}. Later, Niezgoda \cite{2} used this transform to obtain certain Cassel- type inequalities.

Our primary goal in this work is to use the transform $C_{M,m}$ to obtain new operator inequalities that involve relations among $|A|,|A^*|, \mathfrak{R}A$ and $\mathfrak{I}A$, where $|A|=(A^*A)^{\frac{1}{2}}$ and $\mathfrak{R}A$ and $\mathfrak{I}A$ refer respectively to the real and imaginary parts of the operator $A$. Then new forms of numerical radius inequalities are found using this transform.

To this end, we need to remind the reader of some terminologies. Recall that an operator $A\in\mathcal{B}(\mathcal{H})$ is said to be positive semi-definite if $\left<Ax,x\right>\geq 0$ for all vectors $x\in\mathcal{H}$ and we write $A\geq 0$, while it is said to be positive (or positive definite) if $\left<Ax,x\right> >0$ for all nonzero $x\in\mathcal{H},$ and we write $A>0$. The real and imaginary parts of the operator $A$ are defined respectively by $\mathfrak{R}A=\frac{A+A^*}{2}$ and $\mathfrak{I}A=\frac{A-A^*}{2\textup i}.$ Further, $A$ is said to be accretive (dissipative) if $\mathfrak{R}A\geq 0$ ($\mathfrak{I}A\geq 0$).  If $\mathfrak{R}A,\mathfrak{I}A\geq 0$, then $A$ is said to be accretive-dissipative. Accretive, dissipative, and accretive-dissipative operators have received considerable attention in the literature due to their applicability in operator theory and its inequalities. We refer the reader to \cite{bed_pos,bed_results,Drury1,drury,7,Lin 1,raissouli} as a list of references dealing with such operators.

Our approach here will assume accretivity or dissipativity of the transform $C_{M,m}.$ For this, we begin by presenting simple properties of this transform for this context.

\begin{proposition}\label{prop_properties_C}
Let $m<M$ be given real numbers and let $C_{M,m}:\mathcal{B}(\mathcal{H})\to \mathcal{B}(\mathcal{H})$ be the transform $C_{M,m}(A)=(MI-A^*)(A-mI)$. Then
\begin{equation}
         {C_{M,m}}\left( {{A}^{*}} \right)=C_{M,m}^{*}\left( A \right)\quad\Leftrightarrow \quad A\text{ is normal}.     \tag{1}
\end{equation}

\begin{equation}
{C_{M,m}}\left( A \right)=C_{M,m}^{*}\left( A \right)\quad\Leftrightarrow \quad A\text{ is self-adjoint}. \tag{2}
\end{equation}

\begin{equation}
{C_{M,m}}\left( \left| A \right| \right)\text{ is accretive }\Leftrightarrow m{I}\le \left| A \right|\le M{I}. \tag{3}
\end{equation} 

\begin{equation}
\mathfrak R{C_{M,m}}\left( \textup i{{A}^{*}} \right),\mathfrak R{C_{M,m}}\left( A \right)\le {{\left( \frac{M-m}{2} \right)}^{2}}{I}. \tag{4}
\end{equation} 

\begin{equation}
\mathfrak I{C_{M,m}}\left( A \right)\ge 0 \; if\; and \;only\; if\; \mathfrak IA\ge 0. \tag{5}
\end{equation}

\begin{equation}
If \;\mathfrak R {C_{M,m}}\left( A \right)\ge 0,\; then \;\mathfrak RA\ge 0. \tag{6}
\end{equation}

\begin{equation}
If\; {C_{M,m}}\left( A \right)\; is\; accretive-dissipative,\; then \;A\; is\; accretive-dissipative. \tag{7}
\end{equation}

\end{proposition}
\begin{proof}
Statement (1) follows, noting that
\[C_{M,m}^{*}\left( A \right)-{C_{M,m}}\left( {{A}^{*}} \right)={{\left| {{A}^{*}} \right|}^{2}}-{{\left| A \right|}^{2}},\]
while (2) follows immediately from
\[{C_{M,m}}\left( A \right)-C_{M,m}^{*}\left( A \right)=\left( M-m \right)\left( A-{{A}^{*}} \right).\]
The third statement follows from the following fact \[\left( M{I}-\left| A \right| \right)\left( \left| A \right|-m{I} \right)\ge 0\Leftrightarrow m{I}\le \left| A \right|\le M{I}.\]

 It is not hard to check that
 \begin{equation}\label{5}
\mathfrak R{C_{M,m}}\left( A \right)+{{\left| A-\frac{M+m}{2}{I} \right|}^{2}}={{\left( \frac{M-m}{2} \right)}^{2}}{I},
 \end{equation}
and
\[\mathfrak R{C_{M,m}}\left( \textup i{{A}^{*}} \right)+{{\left| \textup i{{A}^{*}}-\frac{M+m}{2}{I} \right|}^{2}}={{\left( \frac{M-m}{2} \right)}^{2}}{I},\]
which together imply (4). On the other hand, direct calculations show that 
\[\mathfrak I{C_{M,m}}\left( A \right)=\left( M-m \right)\mathfrak IA\] which implies (5), and (6) follows from the definition of $C_{M,m}$ likewise. The last statement (7) follows from both (5) and (6).

\end{proof}
Having established these fundamental properties, we proceed to our main results in the coming sections, where operator inequalities are discussed first, then numerical radius inequalities are presented.

\section{Operator Inequalities}
In this section, we present several operator inequalities using properties of the transform $C_{M,m}$. In particular, this section's results will focus on relations among  $|A|,|A^*|,\mathfrak{R}A$,  and $\mathfrak{I}A$. We should remark that, in general, such relations do not exist. However, by imposing an extra condition on $C_{M,m}$, we obtain such ties.\\
We notice the appearance of the constant $\frac{M+m}{2\sqrt{Mm}}$, which is the ratio between the arithmetic and geometric means of $M$ and $m$. This constant is, in fact, the square root of the well-known Kantorovich constant. In the sequel, $m$ and $M$ are positive numbers.
\begin{theorem}\label{00}
Let $A\in \mathcal B\left( \mathcal H \right)$.
\begin{itemize}
\item[(i)]  If ${C_{M,m}}\left( A \right)$ is accretive, then
\begin{equation}\label{0}
\left| A \right|\le \frac{M+m}{2\sqrt{Mm}}\mathfrak RA.
\end{equation}

\item[(ii)]  If ${C_{M,m}}\left(\textup i{{A}^{*}} \right)$ is accretive, then
 \[\left| {{A}^{*}} \right|\le \frac{M+m}{2\sqrt{Mm}}\mathfrak IA.\]
 
\item[(iii)]  If $A$ is invertible and  ${C_{M,m}}\left({{A}^{-1}} \right)$ is accretive, then
 \[\left| {{A}^{-1}} \right|\le \frac{M+m}{2\sqrt{Mm}}\mathfrak R{{A}^{-1}}.\]
\end{itemize}
\end{theorem}
\begin{proof}
For the first statement, the assumption implies that $\mathfrak R {C_{M,m}}\left( A \right)\ge 0$. This is equivalent to saying
	\[\frac{\left( MA-Mm{I}-{{\left| A \right|}^{2}}+m{{A}^{*}} \right)+{{\left( MA-Mm{I}-{{\left| A \right|}^{2}}+m{{A}^{*}} \right)}^{*}}}{2}\ge 0.\]
Namely,
\begin{equation}\label{04}
\left( M+m \right)\mathfrak RA\ge Mm{I}+{{\left| A \right|}^{2}}.
\end{equation}
Applying the operator arithmetic-geometric mean inequality, we infer that
	\[Mm{I}+{{\left| A \right|}^{2}}\ge 2\sqrt{Mm}\left| A \right|.\]
Combining the last two inequalities, we get \eqref{0}.

To prove parts (ii) and (iii), we replace, in part (i), $A$ by $\textup i{{A}^{*}}$, and $A$ by ${{A}^{-1}}$, respectively.
\end{proof}
 
We note that the inequality \eqref{0} in Theorem \ref{00} has been given in \cite[Proposition 2.4]{2}, using a different method.

Theorem \ref{00} entails the following reverse of the triangle inequality.
\begin{theorem}\label{thm_block_1}
Let $S,T\in \mathcal B\left( \mathcal H \right)$. If ${C_{M,m}}\left(  \left[ \begin{matrix}
   0 & S  \\
   {{T}^{*}} & 0  \\
\end{matrix} \right]  \right)$ is accretive, then

\[\left\| S \right\|+\left\| T \right\|+\left| \; \left\| S \right\|-\left\| T \right\| \; \right|\le \frac{M+m}{2\sqrt{Mm}}\left\| S+T \right\|.\]
\end{theorem}
\begin{proof}
Let $A=\left[ \begin{matrix}
   0 & S  \\
   {{T}^{*}} & 0  \\
\end{matrix} \right]\in\mathcal{B}(\mathcal H\oplus \mathcal H)$. If ${C_{M,m}}\left( A  \right)$ is accretive, then
\[\begin{aligned}
   \left\| S+T \right\|&=\left\| A+{{A}^{*}} \right\| \\
      &=2\|\mathfrak{R}A\|\\ 
 & \ge \frac{4\sqrt{Mm}}{M+m}\left\| A \right\| \quad \text{(by \eqref{0},\;since\;$\|A\|=\|\;|A|\;\|$)}\\ 
 & =\frac{4\sqrt{Mm}}{M+m}\left\| \left[ \begin{matrix}
   0 & S  \\
   {{T}^{*}} & 0  \\
\end{matrix} \right] \right\| \\ 
 & =\frac{4\sqrt{Mm}}{M+m}\max \left( \left\| S \right\|,\left\| T \right\| \right)\quad\text{(since $\|X\|=\|\;|X^*|\;\|$)}.  
\end{aligned}\]
Noting the identity $\max\{a,b\}=\frac{a+b+|a-b|}{2}$,  the desired inequality follows.
\end{proof}

An upper bound of the difference $|A|-\mathfrak{R}A$ is given next.
\begin{corollary}
Let $A\in \mathcal B\left( \mathcal H \right)$.
\begin{itemize}
\item[(i)]  If ${C_{M,m}}\left( A \right)$ is accretive, then
 \[0\leq \left| A \right|-\mathfrak RA\le \frac{{{\left( \sqrt{M}-\sqrt{m} \right)}^{2}}}{M+m}\left\| A \right\|{I}.\]

\item[(ii)]  If ${C_{M,m}}\left(\textup i{{A}^{*}} \right)$ is accretive, then
\[0\leq \left| {{A}^{*}} \right|-\mathfrak IA\le \frac{{{\left( \sqrt{M}-\sqrt{m} \right)}^{2}}}{M+m}\left\| A \right\|{I}.\]
 
\item[(iii)]  If $A$ is invertible and ${C_{M,m}}\left({{A}^{-1}} \right)$ is accretive, then
\[0\leq \left| A^{-1} \right|-\mathfrak RA^{-1}\le \frac{{{\left( \sqrt{M}-\sqrt{m} \right)}^{2}}}{M+m}\left\| A^{-1} \right\|{I}.\]
\end{itemize}
\end{corollary}
 \begin{proof}
By Theorem \ref{00}, 
 \[\begin{aligned}
    \left| A \right|-\mathfrak RA&\le \left( 1-\frac{2\sqrt{Mm}}{M+m} \right)\left| A \right| \\ 
  & =\frac{{{\left( \sqrt{M}-\sqrt{m} \right)}^{2}}}{M+m}\left| A \right| \\ 
  & \le \frac{{{\left( \sqrt{M}-\sqrt{m} \right)}^{2}}}{M+m}\left\|\; \left| A \right| \;\right\|{I} \\ 
  & =\frac{{{\left( \sqrt{M}-\sqrt{m} \right)}^{2}}}{M+m}\left\| A \right\|{I}.  
 \end{aligned}\]
This completes the proof of part (i). 

The other two parts can be proven similarly.
 \end{proof} 

In the following theorem, we use the fact that the function $f(x)=x^2$ is an operator convex function on $\mathbb{R}$. This means that when $A$ and $B$ are self adjoint operators in $\mathcal{B}(\mathcal{H})$, we have $((1-t)A+tB)^2\leq (1-t)A^2+tB^2$ for $0\leq t\leq 1.$ We refer the reader to \cite[Chapter V]{bhatia_matrix}  for
further information on operator convex functions.
\begin{theorem}
Let $A\in \mathcal B\left( \mathcal H \right)$.
 If ${C_{M,m}}\left( A \right)$ and  ${C_{M,m}}\left(\textup i{{A}^{*}} \right)$ are accretive, then for any $0\le t\le 1$,
 \[\left( 1-t \right)\left| {{A}^{*}} \right|+t\left| A \right|\le \frac{M+m}{2\sqrt{Mm}}\left( \left( 1-t \right)\mathfrak IA+t\mathfrak RA \right).\]
\end{theorem}
\begin{proof}
By  (i) and (ii) of Theorem \ref{00}, we have
\[\left( 1-t \right)Mm{I}+\left( 1-t \right){{\left| {{A}^{*}} \right|}^{2}}\le \left( 1-t \right)\left( M+m \right)\mathfrak IA\]
and
\[tMm{I}+t{{\left| A \right|}^{2}}\le t\left( M+m \right)\mathfrak RA.\]
Adding these inequalities, we obtain
\[Mm{I}+\left( 1-t \right){{\left| {{A}^{*}} \right|}^{2}}+t{{\left| A \right|}^{2}}\le \left( M+m \right)\left( \left( 1-t \right)\mathfrak IA+t\mathfrak RA \right).\]
Now, using the operator convexity of the function $f\left( t \right)={{t}^{2}}$ on $\left( 0,\infty  \right)$ and the operator arithmetic-geometric mean inequality, we get
\[\begin{aligned}
   2\sqrt{Mm}\left( \left( 1-t \right)\left| {{A}^{*}} \right|+t\left| A \right| \right)&\le Mm{I}+{{\left( \left( 1-t \right)\left| {{A}^{*}} \right|+t\left| A \right| \right)}^{2}} \\ 
 & \le Mm{I}+\left( 1-t \right){{\left| {{A}^{*}} \right|}^{2}}+t{{\left| A \right|}^{2}}.  
\end{aligned}\]
This completes the proof.
\end{proof}

Squaring operator inequalities are more complex than squaring real inequalities. In other words, if $a,b$ are positive numbers such that $a\leq b,$ then $a^2\leq b^2$. Now, if $A\leq B$, where $A, B$ are positive operators,  we cannot conclude $A^2\leq B^2$ since the function $f(x)=x^2$ is not operator monotone. We refer the reader to \cite[Chapter V]{bhatia_matrix} to get more insight about this. The following result shows that the inequalities in Theorem \ref{00} can be squared.
\begin{theorem}\label{05}
Let $A\in \mathcal B\left( \mathcal H \right)$.
\begin{itemize}
\item[(i)]  If ${C_{M,m}}\left( A \right)$ is accretive, then
\begin{equation}\label{06}
{{\left| A \right|}^{2}}\le {{\left( \frac{M+m}{2\sqrt{Mm}} \right)}^{2}}{{\left( \mathfrak RA \right)}^{2}}.
\end{equation}

\item[(ii)]  If ${C_{M,m}}\left(\textup i{{A}^{*}} \right)$ is accretive, then
\[\left| {{A}^{*}} \right|^{2}\le {{\left( \frac{M+m}{2\sqrt{Mm}} \right)}^{2}}{{\left( \mathfrak IA \right)}^{2}}.\]

\item[(iii)]  If $A$ is invertible and ${C_{M,m}}\left( A^{-1} \right)$ is accretive, then
\begin{equation*}
{{\left| A^{-1} \right|}^{2}}\le {{\left( \frac{M+m}{2\sqrt{Mm}} \right)}^{2}}{{\left( \mathfrak RA^{-1} \right)}^{2}}.
\end{equation*}
\end{itemize}
\end{theorem}
\begin{proof}
By the inequality \eqref{04},
	\[{{\left| A \right|}^{2}}\le \left( M+m \right)\mathfrak RA-Mm{I}.\]
So, to prove the inequality \eqref{06}, it is enough to show that
	\[\left( M+m \right)\mathfrak RA-Mm{I}\le {{\left( \frac{M+m}{2\sqrt{Mm}} \right)}^{2}}{{\left( \mathfrak RA \right)}^{2}},\]
holds. For $m\leq t\leq M$, define 
	\[f\left( t \right)={{\left( \frac{M+m}{2\sqrt{Mm}} \right)}^{2}}{{t}^{2}}-\left( M+m \right)t+Mm.\]
Then
	\[f'\left( t \right)=\frac{{{\left( M+m \right)}^{2}}}{2Mm}t-\left( M+m \right),\]
and
	\[f''\left( t \right)=\frac{{{\left( M+m \right)}^{2}}}{2Mm}>0.\]
Namely, $f$ is convex. On the other hand, if we put $f'\left( t \right)=0$, then we get $t={2Mm}/{\left( M+m \right)}\;$, and $f\left( {2Mm}/{\left( M+m \right)}\; \right)=0$. So $f\left( t \right)$ is positive, i.e.,
	\[\left( M+m \right)t-Mm\le {{\left( \frac{M+m}{2\sqrt{Mm}} \right)}^{2}}{{t}^{2}}.\]
We get the desired result by applying functional calculus for the positive operator $\mathfrak RA$.

The other parts can be established similarly, so we omit the details.
\end{proof}

In fact, Theorem \ref{00} is a direct consequence of Theorem \ref{05}, since $f\left( t \right)={{t}^{{1}/{2}\;}}$ is operator monotone on $\left( 0,\infty  \right)$, \cite[Theorem 1.5.9]{5}.

Our next target is to investigate commutators of $|A|$ and $\mathfrak{R}A$. To this end, the following lemma will be needed.
\begin{lemma}\label{3}
\cite[Lemma 2.1]{8} Let $A,B\in \mathcal B\left( \mathcal H \right)$ be positive operators and let $\alpha >0$. Then 
	\[A\le \alpha B\Leftrightarrow \left\| {{A}^{\frac{1}{2}}}{{B}^{-\frac{1}{2}}} \right\|\le \sqrt{\alpha }.\]
\end{lemma}
\begin{corollary}\label{4}
Let $A\in \mathcal B\left( \mathcal H \right)$ be such that both $\mathfrak{R}A$ and $\mathfrak{I}A$ are invertible.
\begin{itemize}
\item[(i)]  If ${C_{M,m}}\left( A \right)$ is accretive, then
\[\left| \left| A \right|{{\left( \mathfrak RA \right)}^{-1}}+{{\left( \mathfrak RA \right)}^{-1}}\left| A \right| \right|\le \frac{M+m}{\sqrt{Mm}}{I},\]
and
\[\left| A \right|{{\left( \mathfrak RA \right)}^{-1}}+{{\left( \mathfrak RA \right)}^{-1}}\left| A \right|\le \frac{M+m}{\sqrt{Mm}}{I}.\]
\item[(ii)]  If ${C_{M,m}}\left(\textup i{{A}^{*}} \right)$ is accretive, then
\[\left| \left| {{A}^{*}} \right|{{\left( \mathfrak IA \right)}^{-1}}+{{\left( \mathfrak IA \right)}^{-1}}\left| {{A}^{*}} \right| \right|\le \frac{M+m}{\sqrt{Mm}}{I},\]
and
\[\left| {{A}^{*}} \right|{{\left( \mathfrak IA \right)}^{-1}}+{{\left( \mathfrak IA \right)}^{-1}}\left| {{A}^{*}} \right|\le \frac{M+m}{\sqrt{Mm}}{I}.\]
\end{itemize}
\end{corollary}
\begin{proof}
By Lemma \ref{3}, the inequality \eqref{06} is equivalent to
\[\left\| \left| A \right|{{\left( \mathfrak RA \right)}^{-1}} \right\|\le \frac{M+m}{2\sqrt{Mm}}.\]
By \cite[Lemma 3.5.12]{4}, we get
\[\left[ \begin{matrix}
   \frac{M+m}{2\sqrt{Mm}}{I} & \left| A \right|{{\left( \mathfrak RA \right)}^{-1}}  \\
   {{\left( \mathfrak RA \right)}^{-1}}\left| A \right| & \frac{M+m}{2\sqrt{Mm}}{I}  \\
\end{matrix} \right]\ge 0\text{ and }\left[ \begin{matrix}
   \frac{M+m}{2\sqrt{Mm}}{I} & {{\left( \mathfrak RA \right)}^{-1}}\left| A \right|  \\
   \left| A \right|{{\left( \mathfrak RA \right)}^{-1}} & \frac{M+m}{2\sqrt{Mm}}{I}  \\
\end{matrix} \right]\ge 0.\]
Adding these two operator matrices, we have
\[\left[ \begin{matrix}
   \frac{M+m}{\sqrt{Mm}}{I} & \left| A \right|{{\left( \mathfrak RA \right)}^{-1}}+{{\left( \mathfrak RA \right)}^{-1}}\left| A \right|  \\
   {{\left( \mathfrak RA \right)}^{-1}}\left| A \right|+\left| A \right|{{\left( \mathfrak RA \right)}^{-1}} & \frac{M+m}{\sqrt{Mm}}{I}  \\
\end{matrix} \right]\ge 0.\]
This completes the proof.
\end{proof}

\begin{remark}
To show how Corollary \ref{4} improves Theorem \ref{05}, we recall the following inequality \cite[(1.9)]{9}, for  $A,B\geq 0$:
\begin{equation}\label{ned_bhakit}
||AB||\leq\frac{1}{4}\|(A+B)\|^2.
\end{equation}
 Now, we notice that
\[\begin{aligned}
   4\left| A \right|{{\left( \mathfrak RA \right)}^{-2}}\left| A \right|&\le 4\left\| \left| A \right|{{\left( \mathfrak RA \right)}^{-2}}\left| A \right| \right\|{I} \\ 
 & \le {{\left\| \left| A \right|{{\left( \mathfrak RA \right)}^{-1}}+{{\left( \mathfrak RA \right)}^{-1}}\left| A \right| \right\|}^{2}}{I} \quad \text{(by \eqref{ned_bhakit})}\\ 
 & \le {{\left( \frac{M+m}{\sqrt{Mm}} \right)}^{2}}{I}, 
\end{aligned}\]
which is equivalent to saying that
\[\begin{aligned}
   {{\left( \mathfrak RA \right)}^{-2}}&\le \left\| \left| A \right|{{\left( \mathfrak RA \right)}^{-2}}\left| A \right| \right\|{{\left| A \right|}^{-2}} \\ 
 & \le {{\left\| \frac{\left| A \right|{{\left( \mathfrak RA \right)}^{-1}}+{{\left( \mathfrak RA \right)}^{-1}}\left| A \right|}{2} \right\|}^{2}}{{\left| A \right|}^{-2}} \\ 
 & \le {{\left( \frac{M+m}{2\sqrt{Mm}} \right)}^{2}}{{\left| A \right|}^{-2}}.  
\end{aligned}\]
Now, by taking the inverse, we get
\[\begin{aligned}
   {{\left( \mathfrak RA \right)}^{2}}&\ge {{\left\| \left| A \right|{{\left( \mathfrak RA \right)}^{-2}}\left| A \right| \right\|}^{-1}}{{\left| A \right|}^{2}} \\ 
 & \ge {{\left\| \frac{\left| A \right|{{\left( \mathfrak RA \right)}^{-1}}+{{\left( \mathfrak RA \right)}^{-1}}\left| A \right|}{2} \right\|}^{-2}}{{\left| A \right|}^{2}} \\ 
 & \ge {{\left( \frac{2\sqrt{Mm}}{M+m} \right)}^{2}}{{\left| A \right|}^{2}}.  
\end{aligned}\]
\end{remark}

For an arbitrary $A\in\mathcal{B}(\mathcal{H})$, the inequality
\begin{equation}\label{1}
\Phi^{\frac{1}{2}}(|A|^2)\geq \Phi(|A|)
\end{equation} 
is well known for the unital positive linear map $\Phi:\mathcal{B}(\mathcal{H})\to \mathcal{B}(\mathcal{H})$, \cite{choi,davis}. In this context, recall that such a map is a map that satisfies $\Phi(A)\geq 0$ whenever $A\geq 0$ and $\Phi(I)=I.$ In what follows, a reversed version is presented via the transform $C_{M,m}.$
\begin{lemma}\label{2}
 Let $A\in \mathcal B\left( \mathcal H \right)$ and let $\Phi $ be a unital positive linear map on $\mathcal B\left( \mathcal H \right)$. If ${C_{M,m}}\left( \left| A \right| \right)$  is accretive, then
	\[{{\Phi }^{\frac{1}{2}}}\left( {{\left| A \right|}^{2}} \right)\le \frac{M+m}{2\sqrt{Mm}}\Phi \left( \left| A \right| \right).\]
\end{lemma}
\begin{proof}
Since ${C_{M,m}}\left( \left| A \right| \right)$  is accretive, we have by \cite[Theorem 1.32 (iii)]{FMPS},
\[\Phi \left( {{\left| A \right|}^{2}} \right)\le \frac{{{\left( M+m \right)}^{2}}}{4Mm}{{\Phi }^{2}}\left( \left| A \right| \right).\]
The result follows by taking into account that the function $f\left( t \right)={{t}^{\frac{1}{2}}}$ is operator monotone on $\left( 0,\infty  \right)$.
\end{proof}
In the next theorem, reverses of the inequalities of Theorem \ref{00} are presented.
\begin{theorem}\label{7}
Let $A\in \mathcal B\left( \mathcal H \right)$.
\begin{itemize}
\item[(i)] If ${C_{M,m}}\left( \left| A \right| \right)$  is accretive, then
\[\mathfrak RA\le \frac{M+m}{2\sqrt{Mm}}\left| A \right|.\]

\item[(ii)] If ${C_{M,m}}\left( \left| \textup i A^{*} \right| \right)$  is accretive, then
\[\mathfrak IA\le \frac{M+m}{2\sqrt{Mm}}\left| {{A}^{*}} \right|.\]

\item[(iii)] If $A$ is invertible and ${C_{M,m}}\left( \left| A^{-1} \right| \right)$  is accretive, then
\[\mathfrak RA^{-1}\le \frac{M+m}{2\sqrt{Mm}}\left| A^{-1} \right|.\]
\end{itemize}
\end{theorem}
\begin{proof}
We prove part (i) since the other parts are easy to prove. By the inequality \eqref{1}, we have for any unit vector $x\in \mathcal H$,
\[\left| \left\langle Ax,x \right\rangle  \right|\le {{\left\langle {{\left| A \right|}^{2}}x,x \right\rangle }^{\frac{1}{2}}}.\]
Applying Lemma \ref{2} for $\Phi \left( T \right)=\left\langle Tx,x \right\rangle {I} \left( x\in \mathcal H,\left\| x \right\|=1 \right)$, implies
 \[{{\left\langle {{\left| A \right|}^{2}}x,x \right\rangle }^{\frac{1}{2}}}\le \frac{M+m}{2\sqrt{Mm}}\left\langle \left| A \right|x,x \right\rangle .\]
 Hence,
 \[\left| \left\langle Ax,x \right\rangle  \right|\le \frac{M+m}{2\sqrt{Mm}}\left\langle \left| A \right|x,x \right\rangle .\]
 Now, by combining this inequality with the fact that $\left\langle \mathfrak RAx,x \right\rangle  \le \left| \left\langle Ax,x \right\rangle  \right|$, we reach the desired result.
\end{proof}

\begin{corollary}
Let $A\in \mathcal B\left( \mathcal H \right)$.
\begin{itemize}
\item[(i)] If ${C_{M,m}}\left( \left| A \right| \right)$  is accretive, then
\[\mathfrak RA-\left| A \right|\le \frac{{{\left( M-m \right)}^{2}}}{2\sqrt{Mm}}\left\| A \right\|{I}.\]

\item[(ii)] If ${C_{M,m}}\left( \left| \textup i A^{*} \right| \right)$  is accretive, then
\[\mathfrak IA-\left| {{A}^{*}} \right|\le \frac{{{\left( M-m \right)}^{2}}}{2\sqrt{Mm}}\left\| A \right\|{I}.\]

\end{itemize}
\end{corollary}

 \section{Numerical Radius Inequalities}
This section uses the properties of the transform $C_{M,m}$ and its consequences to obtain some new numerical radius inequalities. In this context, we recall that the numerical range of  $A\in\mathcal{B}(\mathcal{H})$ is defined as
$$W(A)=\{\left<Ax,x\right>:x\in\mathcal{H},\|x\|=1\}.$$

Then, we define the numerical radius of  $A$ as 
$$\omega(A)=\sup\{|z|:z\in W(A)\}.$$
 The numerical radius has a notable recognition in the literature due to its impact on understanding the geometry of the numerical range of the operator. Among the most basic inequalities of the numerical radius, we have
\begin{equation}\label{eq_omega_norm}
\frac{1}{2}\|A\|\leq \omega(A)\leq \|A\|
\end{equation} 
and
\begin{equation}\label{eq_real_omega}
\|\mathfrak{R}A\|\leq \omega(A) \;{\text{and}}\; \|\mathfrak{I}A\|\leq \omega(A).
\end{equation} 
We refer the reader to \cite{bhumos,7,gaubook} for further information on the numerical radius, its properties, and recent advances in its inequalities.

 We begin with the following observation.

 \begin{remark}
 From Theorem \ref{00}, we have
 \[\left\| A \right\|\le \left\| \frac{M+m}{2\sqrt{Mm}}\mathfrak RA \right\|=\frac{M+m}{2\sqrt{Mm}}\left\| \mathfrak RA \right\|.\]
 That is
 \[\left\| A \right\|\le \frac{M+m}{2\sqrt{Mm}}\left\| \mathfrak RA \right\|\quad\text{ and }\quad\left\| A \right\|-\left\| \mathfrak RA \right\|\le \frac{{{\left( \sqrt{M}-\sqrt{m} \right)}^{2}}}{M+m}\left\| A \right\|.\]
Noting that for any operator $A\in\mathcal{B}(\mathcal{H})$, $\omega(A)\leq\|A\|$ and $\omega(A)\geq \|\mathfrak{R}A\|$,  we deduce that when $C_{M,m}(A)$ is accretive, one has
\begin{equation}\label{eq_w_R_1}
\omega \left( A \right)\le \frac{M+m}{2\sqrt{Mm}}\left\| \mathfrak RA \right\|\quad\text{ and }\quad\omega \left( A \right)-\left\| \mathfrak RA \right\|\le \frac{{{\left( \sqrt{M}-\sqrt{m} \right)}^{2}}}{M+m}\omega \left( A \right),
\end{equation}
 and
 \begin{equation}\label{eq_w_R_2}
 \left\| A \right\|\le \frac{M+m}{2\sqrt{Mm}}\omega \left( A \right)\quad\text{ and }\quad\left\| A \right\|-\omega \left( A \right)\le \frac{{{\left( \sqrt{M}-\sqrt{m} \right)}^{2}}}{M+m}\left\| A \right\|.
 \end{equation}

The last two inequalities in \eqref{eq_w_R_2} have been proved in \cite[Remark 35]{3}. Before proceeding, it is worth mentioning the importance of the above inequalities. We know that for any $A\in\mathcal{B}(\mathcal{H})$, $\|\mathfrak{R}A\|\leq \omega(A)$. Thus, the first inequality in \eqref{eq_w_R_1} provides a reversed version of this known inequality. Of course, this is valid when $C_{M,m}$ is accretive. Further, under this condition, the first inequality in \eqref{eq_w_R_2} provides a refinement of the well-known inequality $\|A\|\leq 2\omega(A)$ in case we have $\frac{M+m}{2\sqrt{Mm}}<2.$ Notice that this latter ratio is always not less than one. Here we give a numerical example to show that for the given matrix, the lower bound $\frac{2\sqrt{Mm}}{M+m}\|A\|$ is larger than $\frac{\|A\|}{2};$ as lower bounds of $\omega(A).$ For this, let $$A=\left[\begin{array}{cc} 5-4\textup{i}&2\textup{i}\\1+\textup{i}&6\end{array}\right], m=4, M=50.$$ Then it can be easily seen that
$$C_{M,m}(A)=\left[\begin{matrix} 27-184\textup{i}&6+92\textup{i}\\52+46\textup{i}&84 \end{matrix}\right]\;and\;\mathfrak{R} C_{M,m}(A)=\left[\begin{matrix} 54&54+46\textup{i}\\ 58-46\textup{i}&168\end{matrix}\right].$$
Since $\mathfrak{R} C_{M,m}(A)>0$, it follows that $C_{M,m}(A)$ is accretive. Now direct calculations show that
$$\frac{2\sqrt{Mm}}{M+m}\|A\|=3.56083\;and\;\frac{\|A\|}{2}=3.3991.$$ Thus, in this example, we have
$$\omega(A)\geq \frac{2\sqrt{Mm}}{M+m}\|A\|>\frac{\|A\|}{2},$$ explaining the significance of the first inequality in \eqref{eq_w_R_2}.
 \end{remark}
 In the following, we present an inequality that relates the numerical radius of $A$ with the norms of its real and imaginary parts as a reversed type of \eqref{eq_real_omega}.
 \begin{theorem}
 Let $A\in \mathcal B\left( \mathcal H \right)$.
  If ${C_{M,m}}\left( A \right)$ and  ${C_{M,m}}\left(\textup i{{A}^{*}} \right)$ are accretive, then
\[\omega \left( A \right)\le \frac{M+m}{2\sqrt{Mm}}\sqrt{\left\| \mathfrak RA \right\|\left\| \mathfrak IA \right\|}.\]
 \end{theorem}
 \begin{proof}
Let  $x\in \mathcal H$ be a unit vector. Then by the mixed Schwarz inequality \cite[pp. 75-76]{10}, and Theorem \ref{00}, we have
 \[\begin{aligned}
    \left| \left\langle Ax,x \right\rangle  \right|&\le \sqrt{\left\langle \left| A \right|x,x \right\rangle \left\langle \left| {{A}^{*}} \right|x,x \right\rangle } \\ 
  & \le \frac{M+m}{2\sqrt{Mm}}\sqrt{\left\langle \mathfrak RAx,x \right\rangle \left\langle \mathfrak IAx,x \right\rangle } \\ 
  & \le \frac{M+m}{2\sqrt{Mm}}\sqrt{\left\| \mathfrak RAx \right\|\left\| \mathfrak IAx \right\|} \\ 
  & \le \frac{M+m}{2\sqrt{Mm}}\sqrt{\left\| \mathfrak RA \right\|\left\| \mathfrak IA \right\|}.  
 \end{aligned}\]
 
Now, we get the desired result by taking supremum over all unit vectors $x\in \mathcal H$.
  \end{proof}
  In the following, we present a lower bound of the numerical radius in terms of $\|\;|A|^2+|A^*|^2\|.$ The significance of this result is explained in Remark \ref{rem_kit_1} below.
\begin{theorem}\label{01}
Let $A\in \mathcal B\left( \mathcal H \right)$. If ${C_{M,m}}\left( A \right)$ is accretive, then
\[\frac{2Mm}{{{\left( M+m \right)}^{2}}}\left\| {{\left| A \right|}^{2}}+{{\left| {{A}^{*}} \right|}^{2}} \right\|\le {{\omega }^{2}}\left( A \right).\]
\end{theorem}
\begin{proof}
We know that
\[{{\left\| A \right\|}^{2}}={{\left\| \,\left| A \right|\, \right\|}^{2}}=\left\| {{\left| A \right|}^{2}} \right\|={{\left\| \,\left| {{A}^{*}} \right| \,\right\|}^{2}}=\left\| {{\left| {{A}^{*}} \right|}^{2}} \right\|.\]
This, together with Theorem \ref{00}, implies that
\[\left\| {{\left| A \right|}^{2}}+{{\left| {{A}^{*}} \right|}^{2}} \right\|\le \frac{{{\left( M+m \right)}^{2}}}{2Mm}{{\omega }^{2}}\left( A \right),\]
as desired.
\end{proof}

\begin{remark}\label{rem_kit_1}
If
\[Mm\ge \frac{1}{4}{{\left( M-m \right)}^{2}},\]
then, Theorem \ref{01} improves (see \cite[Theorem 1]{6}) 
\[\frac{1}{4}\left\| {{\left| A \right|}^{2}}+{{\left| {{A}^{*}} \right|}^{2}} \right\|\le {{\omega }^{2}}\left( A \right).\]
If we let $f(x)=x-\frac{(x-1)2}{4}, x\geq 1,$ we can see that $f$ is increasing on $[1,3]$ and is decreasing afterwards. Calculating, we find that $f(x)=0$ when $x=3+2\sqrt{2}$, and that $f\geq 0$ on $[1,3+2\sqrt{2}],$ while it is negative on $[3+2\sqrt{2},\infty).$ Letting $x=\frac{M}{m},$ this means that the condition $Mm\ge \frac{1}{4}{{\left( M-m \right)}^{2}}$ holds when $1<\frac{M}{m}\leq 3+2\sqrt{2}.$
\end{remark}

On the other hand, a submultiplicative inequality for the numerical radius may be shown as follows.
 \begin{corollary}\label{02}
Let $A,B\in \mathcal B\left( \mathcal H \right)$. If ${C_{M,m}}\left( A \right)$ and ${{C}_{N,n}}\left( B \right)$ are accretive, then
\[\omega \left( AB \right)\le \frac{\left( M+m \right)\left( N+n \right)}{4\sqrt{MNmn}}\omega \left( A \right)\omega \left( B \right).\]
 \end{corollary}
\begin{proof}
We have
\[\begin{aligned}
   \omega \left( AB \right)&\le \left\| AB \right\| \\ 
 & \le \left\| A \right\|\left\| B \right\| \\ 
 & \le \frac{\left( M+m \right)\left( N+n \right)}{4\sqrt{MNmn}}\omega \left( A \right)\omega \left( B \right),
\end{aligned}\]
where we have used \eqref{eq_w_R_2} to obtain the last inequality. This completes the proof.
\end{proof}

\begin{remark}
If
\[{{\left( \sqrt{MN}-\sqrt{mn} \right)}^{2}}+{{\left( \sqrt{Mn}-\sqrt{Nm} \right)}^{2}}\le 12\sqrt{MNmn},\]
then, Corollary \ref{02} refines (see \cite[Theorem 2.5-2]{7})  
\[\omega \left( AB \right)\le 4\omega \left( A \right)\omega \left( B \right).\]
\end{remark}
It is well known that when $A,B\in\mathcal{B}(\mathcal{H})$ then $\omega(AB-BA^*)\leq 2\|A\|\omega(B),$ \cite{11}. In the following, we present a refinement of this inequality when $C_{M,m}$ is accretive.
\begin{theorem}\label{6}
Let $A,B\in \mathcal B\left( \mathcal H \right)$. 
\begin{itemize}
\item[(i)] If ${C_{M,m}}\left( A \right)$ is accretive, then
\[\omega \left( AB-B{{A}^{*}} \right)\le \left( M-m \right)\omega \left( B \right).\]
\item[(ii)] If ${C_{M,m}}\left( \textup i{{A}} \right)$  is accretive, then
\[\omega \left( AB+{{B}^{*}}A \right)\le \left( M-m \right)\omega \left( B \right).\]
\end{itemize}
\end{theorem}
\begin{proof}
By the inequality (see \cite{11})
\[\omega \left( AB-B{{A}^{*}} \right)\le 2\left\| A \right\|\omega \left( B \right),\]
and the relation \eqref{5}, we can write
\[\begin{aligned}
   \omega \left( AB-B{{A}^{*}} \right)&=\omega \left( \left( A-\frac{M+m}{2}{I} \right)B-B\left( {{A}^{*}}-\frac{M+m}{2}{I} \right) \right) \\ 
 & \le 2\left\| {{A}^{*}}-\frac{M+m}{2}{I} \right\|\omega \left( B \right) \\ 
 & =2\left\| A-\frac{M+m}{2}{I} \right\|\omega \left( B \right) \\ 
 & \le \left( M-m \right)\omega \left( B \right),
\end{aligned}\]
as desired.

The inequality in part (ii) can be shown similarly, so we omit the details.
\end{proof}
In the following, we give an example to show how Theorem \ref{6} improves the inequality $\omega(AB-BA^*)\leq 2\|A\|\omega(B).$
\begin{example}
Let $A=\left[ \begin{matrix}
   2 & 0  \\
   -1 & 4  \\
\end{matrix} \right]$, $M=8$, and $m=0.01$. A simple calculation shows that
	\[\mathfrak R{C_{M,m}}\left( A \right)=\left[ \begin{matrix}
   \frac{551}{50} & -\frac{1601}{200}  \\
   -\frac{1601}{200} & \frac{399}{25}  \\
\end{matrix} \right]>0.\]
In this case
	\[2\left\| A \right\|\approx 8.31\]
while
	\[M-m=7.99.\]
These values imply that Theorem \ref{6} improves the inequality
	\[\omega \left( AB-B{{A}^{*}} \right)\le 2\left\| A \right\|\omega \left( B \right).\]
\end{example}
We conclude with the following result.
\begin{corollary}
Let $A,B\in \mathcal B\left( \mathcal H \right)$. 
If ${C_{M,m}}\left( A \right)$  and ${C_{M,m}}\left( \textup i{{A}} \right)$  are accretive, then

\[\omega \left( AB \right)+\frac{1}{2}\left| \omega \left( AB+{{B}^{*}}A \right)-\omega \left( AB-{{B}^{*}}A \right) \right|\le \left( M-m \right)\omega \left( B \right).\]
\end{corollary}
\begin{proof}
If we replace $A$ by $\textup{i}A$, in Theorem \ref{6}, we get
\[\omega \left( AB\pm {{B}^{*}}A \right)\le \left( M-m \right)\omega \left( B \right).\]
This implies,
\[\begin{aligned}
  & \omega \left( AB \right)+\frac{1}{2}\left| \omega \left( AB+{{B}^{*}}A \right)-\omega \left( AB-{{B}^{*}}A \right) \right| \\ 
 & \le \frac{1}{2}\left( \omega \left( AB+{{B}^{*}}A \right)+\omega \left( AB-{{B}^{*}}A \right)+\left| \omega \left( AB+{{B}^{*}}A \right)-\omega \left( AB-{{B}^{*}}A \right) \right| \right) \\ 
 & =\max \left\{ \omega \left( AB+{{B}^{*}}A \right),\omega \left( AB-{{B}^{*}}A \right) \right\} \\ 
 & \le \left( M-m \right)\omega \left( B \right),
\end{aligned}\]
as desired.
\end{proof}

\section*{Declaration}
The authors declare that they have no competing interests. 
\section*{Data availability}
No data sets are associated with this work.

{\tiny (I. H. G\"um\"u\c s) Department of Mathematics, Faculty of Arts and Sciences, Ad\i yaman University, Ad\i yaman, Turkey}

{\tiny \textit{E-mail address:} igumus@adiyaman.edu.tr}

\vskip 0.3 true cm

{\tiny (H. R. Moradi) Department of Mathematics, Mashhad Branch, Islamic Azad University, Mashhad, Iran
	
	\textit{E-mail address:} hrmoradi@mshdiau.ac.ir}

{\tiny (M. Sababheh) Vice president, Princess Sumaya University for Technology, Amman, Jordan

\textit{E-mail address:} sababheh@psut.edu.jo; sababheh@yahoo.com}

\end{document}